\documentclass[preprint]{elsarticle}
\usepackage{amssymb,amsthm}
\usepackage{tikz}
\usepackage{color}
\newtheorem{thm}{Theorem}[section]
\newtheorem{cor}[thm]{Corollary}
\newtheorem{lem}[thm]{Lemma}
\newtheorem{prop}[thm]{Proposition}
\newdefinition{defn}{Definition}
\newdefinition{rmk}{Remark}
\newdefinition{examp}{Example}

\journal{}

\begin{document}
\begin{frontmatter}
\title{A note on extensions of nilpotent Lie algebras of type 2}%
\author[ur]{Pilar Benito}%
\ead{pilar.benito@unirioja.es}
\author[ur]{Daniel de-la-Concepci\'on}
\ead{daniel-de-la.concepcion@alum.unirioja.es}
\address[ur]{Dpto. Matem\'aticas y Computaci\'on, Universidad de La Rioja, 26004, Logro\~no, Spain}%
\begin{abstract}
We propose the study and description of the structure of complex Lie algebras with nilradical a nilpotent Lie algebra of type 2 by using $\mathfrak{sl}_2(\mathbb{C})$-representation theory. Our results will be applied to review the classification given in \cite{An11} (J. Geometry and Physics, 2011) of the Lie algebras with nilradical the quasiclassical algebra $\mathcal{L}_{5,3}$. A non-Lie algebra has been erroneously included in this classification. The $5$-dimensional Lie algebra $\mathcal{L}_{5,3}$ is a free nilpotent algebra of type $2$ and it is one of two free nilpotent algebras  admitting an invariant metric. According to \cite{Ok98}, quasiclassical algebras  let construct consistent Yang-Mills gauge theories.
\end{abstract}
\end{frontmatter}

\section{Introduction}

From Levi Theorem, any finite-dimensional complex Lie algebra $\mathfrak{g}$ (product denoted by $[x,y]$) with (solvable) radical $\mathfrak{r}$ is of the form $\mathfrak{g}=\mathfrak{s}\oplus \mathfrak{r}$, where $\mathfrak{s}$ is a semisimple subalgebra named \emph{Levi subalgebra} (also termed \emph{Levi factor}). The Lie algebra $\mathfrak{g}$ is called faithful if the adjoint representation of the subalgebra $\mathfrak{s}$ on $\mathfrak{r}$ is faithful (equivalently, $\mathfrak{g}$ contains no nonzero semisimple ideals). Note  that any Lie algebra with radical $\mathfrak{r}$ can be decomposed into the direct sum (as ideals) of a semisimple Lie algebra and a faithful Lie algebra with the same radical. Let now consider the nilradical $\mathfrak{n}$ of $\mathfrak{g}$, the largest nilpotent ideal inside $\mathfrak{g}$. The nilradical is contained in the solvable radical and they satisfy the nice product relation:
\begin{equation}\label{R2inN}[\mathfrak{g},\mathfrak{r}]\subseteq \mathfrak{n}.
\end{equation}In other words, the action of $\mathfrak{g}$ on $\displaystyle\frac{\mathfrak{r}}{\mathfrak{n}}$ is trivial.

Corcerning the problem of classification of Lie algebras of a given radical, in 1944, Malcev \cite{Mal44} (see \cite[Theorem 4.4, section 4]{OnVi94}) stablished the following structure result:

\begin{thm}\label{Malcev44}Any faithful Lie algebra $\mathfrak{g}$ with radical $\mathfrak{r}$ is isomorphic to a Lie algebra of the form $\mathfrak{s}\oplus_{id} \mathfrak{r}$ where $\mathfrak{s}$ is a semisimple subalgebra of a Levi subalgebra $\mathfrak{s}_0$ of the derivation algebra of $\mathfrak{r}$, named $\mathrm{Der}\,\mathfrak{r}$. Moreover, given $\mathfrak{s}_1$ and $\mathfrak{s}_2$ semisimple subalgebras of $\mathfrak{s}_0$, the algebras $\mathfrak{s}_1\oplus_{id} \mathfrak{r}$ and $\mathfrak{s}_2\oplus_{id} \mathfrak{r}$ are isomorphic if and only if $\mathfrak{s}_2=A\mathfrak{s}_1A^{-1}$, where $A$ is an authomorphism of $\mathfrak{r}$.  \hfill $\square$
\end{thm}
Previous theorem reduces the classification problem of Lie algebras with a given radical $\mathfrak{r}$, to the analysis of derivations and automorphisms of the solvable Lie algebra $\mathfrak{r}$. The same argument can be apply if we consider the problem of classifying Lie algebras with a given nilradical. According to \cite{Mal45}, solvable Lie algebras can be classify through nilpotent ones and, from \cite{Sa70}, any nilpotent Lie algebra can be obtained as a quotient of a free nilpotent. So, it seems to be quite natural to start classifying nilpotent Lie algebras, their derivations and automorphisms and then to extend nilpotent algebras through suitable derivations sets to solvable and nonsolvable Lie algebras. This is the starting point in \cite{Au68} where Malcev decompositions of Lie algebras are studied, or in \cite{On76} where the problem of classifying Lie algebras whose radical is just the nilradical of a parabolic subalgebra of a semisimple Lie algebra is treated. In this paper we will deal with free nilpotent Lie algebras of type 2 (i.e.: the codimension of the derived algebra is $2$) and their extensions; from \cite{Sa70}, this class of algebras encodes the classification of filiforms and (type $2$) quasi-filiform Lie algebras. The Levi subalgebra of the  derivation algebra of any free nilpotent Lie algebra of type 2 is just a $3$-dimensional split simple Lie algebra $\mathfrak{sl}_2(\mathbb{C})$. Our main tactic to get extensions is based on the well known representation theory of this simple Lie algebra. This technic has been previously used in \cite{Tu92} to get the classification of the $9$-dimensional nonsolvable indecomposable Lie algebras.

The paper splits into four Sections apart from the Introduction. In Section 2, some basic facts on free nilpotent Lie algebras $\mathfrak{n}_{2,t}$, type 2 and arbitrary nilindex $t$, are given; this section also includes a complete description of their derivation Lie algebras $\mathrm{Der}\,\mathfrak{n}_{2,t}$. As a enlighten application, we present the classification of nilpotent Lie algebras of type 2 and nilindex $t\leq 4$ in a nested way. Section 3 deals with the structure of solvable Lie algebras with nilradical $\mathfrak{n}_{2,t}$. It turns out that the algebras in this class appear as extensions through vector subspaces of $\mathrm{Der}\,\mathfrak{n}_{2,t}$ of dimension at most two. Section 4 is devoted to non-solvable Lie algebras with nilradical $\mathfrak{n}_{2,t}$; this type of extensions have dimension at most one and are  determined by derivations that centralize the Levi subalgebra of $\mathrm{Der}\,\mathfrak{n}_{2,t}$. The results in Sections 3 and 4 will be used in the final section 5 to review and present in a unify way the classifications of Lie algebras with nilradical $\mathfrak{n}_{2,2}$ and $\mathfrak{n}_{2,3}$ given in \cite{RuWi93} (only solvable extensions are considered) and \cite{An11}. The Lie algebra $\mathfrak{n}_{2,2}$ is the  $3$-dimensional Heisenberg algebra and $\mathfrak{n}_{2,3}$ is a $5$-dimensional quasi-classical Lie algebra denoted in \cite{An11} as $\mathcal{L}_{5,3}$. Here quasi-classical means endowed with a symmetric non-degenerate invariant form according to \cite[Definition 2.1]{Ok79}. In \cite{BaOv11}, the authors stablished that $\mathfrak{n}_{2,3}$ and $\mathfrak{n}_{3,2}$ (free nilpotent of type $3$) are the unique quasi-classical free nilpotent Le algebras and, following \cite{Ok98}, quasi-classical Lie algebras let construct consistent Yang-Mills gauge theories. Hence, the results in this paper could have general interest in theorical physics.

Along the paper we deal with complex Lie algebras (complex field is denoted by $\mathbb{C}$), although most part of the results are valid over arbitrary fields of characteristic zero except those included on Section 3, where algebraically closed field features are used. Basics definitions and known facts on Lie algebras follows from \cite{Hu72} and \cite{Ja62}.


\section{Lie algebras of type $2$}

\noindent (The results in this section  are partially included in \cite[Section 3]{BeCo13} where universal free nilpotent Lie algebras of characteristic zero and of arbitrary type are treated. So, the results in this section hold over any field of characteristic zero.)

A nilpotent Lie algebra $\mathfrak{n}$ is said to be \emph{$t$-nilpotent} in case $\mathfrak{n}^t\neq 0$ and $\mathfrak{n}^{t+1}= 0$; the \emph{type}  of $\mathfrak{n}$ is given by the natural number $\mathrm{dim}\,\mathfrak{n}-\mathrm{dim}\,\mathfrak{n}^2$. So type $2$ implies $\mathrm{dim}\,\mathfrak{n}=2+\mathrm{dim}\,\mathfrak{n}^2$. This is the first condition that a filiform Lie algebra satisfies; in fact from \cite[Proposition 4]{Sa70} the class of free nilpotent Lie algebras of type $2$ encodes the classification of the filiform class.

From the $2$-element set $\{x_1,x_2\}$ and the (2-dimensional) vector space $\mathfrak{m}=\mathrm{span}\langle x_1,x_2\rangle$, the universal nilpotent Lie algebra $\mathfrak{n}_{2,t}$ is defined as the quotient:
\begin{eqnarray}\label{universal}
\mathfrak{n}_{2,t}=\frac{\mathfrak{FL}_{2}}{\mathfrak{FL}_{2}^{t+1}}
\end{eqnarray}
of the free Lie algebra $\mathfrak{FL}_{2}$ over $2$-elements through the ideal,$$\mathfrak{FL}_{2}^{t+1}=\sum_{j\geq t+1}\ \mathrm{span}\langle [\dots[x_{i_1}x_{i_2}]\dots x_{i_j}]: x_{i_s}=x_1,x_2\rangle.$$For the Lie algebra in (\ref{universal}) we have (cf. \cite[Proposition 4, Proposition 2]{Sa70} and \cite[Proposition 1.4]{Ga73}):

\begin{prop}\label{DerivationN2}The algebra $\mathfrak{n}_{2,t}$ is a quasicyclic $t$-nilpotent Lie algebra of type $2$ and any other $t$-nilpotent Lie algebra  of type $2$ is an homomorphic image of $\mathfrak{n}_{2,t}$. Moreover, the derivation Lie algebra of $\mathfrak{n}_{2,t}$ is:$$
\mathrm{Der}\, \mathfrak{n}_{2,t} = \{\widehat{\delta}: \delta\in  Hom(\mathfrak{m},\mathfrak{m})\}\ \oplus \ \{\widehat{\delta}:\delta \in Hom(\mathfrak{m},\mathfrak{n}_{2,t}^2)\},$$where
\begin{eqnarray}\label{extension}\widehat{\delta}([x_{i_1},x_{i_2},\dots,x_{i_j},\dots x_{i_s}])=\sum_{j=1}^s[x_{i_1},x_{i_2},\dots,\delta(x_{i_j}),\dots x_{i_s}],
\end{eqnarray}
$\delta \in Hom(\mathfrak{m},\mathfrak{m})$ or $\delta \in Hom(\mathfrak{m},\mathfrak{n}_{2,t}^2)$ and  for $s\geq 2$, $[x_{i_1},x_{i_2},\dots,x_{i_j},\dots x_{i_s}]=$\ $[\dots[[x_{i_1}x_{i_2}]\dots ]x_{i_s}]$.  \hfill $\square$
\end{prop}
According to \cite{Le63}, when a nilpotent Lie algebra $\mathfrak{n}$ has a subspace $\mathfrak{u}$ that complements $\mathfrak{n}^2$ and $\mathfrak{u}$ lets $\mathfrak{n}$ be decomposed as a (finite) direct sum of subspaces $\mathfrak{u}^k=[u,u^{k-1}]$, $\mathfrak{n}$ is called \emph{quasicyclic}. So, from Proposition \ref{DerivationN2}, we can assume w.l.o.g.:
\begin{equation}\label{grad}
\mathfrak{n}_{2,t}=\mathfrak{m}\oplus \mathfrak{m}^2\oplus \dots \oplus \mathfrak{m}^t,
\end{equation}
where $\mathfrak{m}=\mathrm{span}\langle x_1,x_2\rangle$ and for $k\geq 2$, $\mathfrak{m}^k=\mathrm{span}\langle [\dots[x_{i_1}x_{i_2}]\dots x_{i_k}]: x_{i_j}=x_1,x_2\rangle$. The direct sum in (\ref{grad}) provides a graded decomposition and, the dimension of each homogeneous component $\mathfrak{m}^s$, is given by the expression ($\mu$ the M$\ddot{o}$ebius function):
\begin{equation}\label{Moebius}
\mathrm{dim}\, \mathfrak{m}^s=\sum_{d\mid s}\mu(d)2^{\frac{s}{d}}.
\end{equation}
Denote by $\mathfrak{gl}(\mathfrak{m})$ the general linear Lie algebra of $\mathbb{C}$-linear maps $\delta:\mathfrak{m}\to \mathfrak{m}$ and, for $j\geq 1$, $\mathrm{Der}_j\, \mathfrak{n}_{2,t}=\{\widehat{\delta}:\delta \in Hom(\mathfrak{m},\mathfrak{m}^j)\} $. So $\mathrm{Der}_1\, \mathfrak{n}_{2,t}=\{\widehat{\delta}: \delta \in \mathfrak{gl}(\mathfrak{m})\}$, is a Lie subalgebra isomorphic to $\mathfrak{gl}(\mathfrak{m})=\mathfrak{sl}(\mathfrak{m})\oplus \mathbb{C}\cdot id_\mathfrak{m}$ (the $3$- dimensional Lie algebra of traceless maps $\mathfrak{sl}(\mathfrak{m})$ plus the identity map) and
\begin{equation}\label{gradder}
\mathrm{Der}\, \mathfrak{n}_{2,t} =\bigoplus_{j=1}^t \mathrm{Der}_j\, \mathfrak{n}_{2,t}.
\end{equation}
Previous decomposition follows the multiplication rule $[d_i, d_k] \in \mathrm{Der}_{i+k-1}\, \mathfrak{n}_{2,t}$, $d_s \in \mathrm{Der}_{s}\, \mathfrak{n}_{2,t}$. Moreover, the derived algebra $\mathrm{Der}_1^0\, \mathfrak{n}_{2,t}=[\mathrm{Der}_1\, \mathfrak{n}_{2,t},\mathrm{Der}_1\, \mathfrak{n}_{2,t}]$ is just $\mathrm{Der}_1^0\, \mathfrak{n}_{2,t}=\{\widehat{\delta}: \delta\in  \mathfrak{sl}(\mathfrak{m})\}$,  a \emph{Levi subalgebra} of $Der\, \mathfrak{n}_{2,t}$ isomorphic to the split $3$-dimensional Lie algebra $\mathfrak{sl}_2(\mathbb{C})$. In this way,
\begin{equation}\label{der1}
\mathrm{Der}_1\, \mathfrak{n}_{2,t} =\mathrm{Der}_1^0\, \mathfrak{n}_{2,t}\oplus \mathbb{C}\cdot id_{2,t}
\end{equation}
 \noindent where $id_{2,t}=\widehat{id_\mathfrak{m}}$, i.e. the extended derivation of $id_\mathfrak{m}$. The solvable radical $\mathfrak{R}_{2,t}$ and the nilradical $\mathfrak{N}_{2,t}$ of the derivation algebra are given by:
\begin{eqnarray}
\mathfrak{N}_{2,t}= \bigoplus_{j\geq2}^t \mathrm{Der}_j\, \mathfrak{n}_{2,t}\\
\mathfrak{R}_{2,t}=\mathbb{C}\cdot id_{2,t}\ \oplus\ \mathfrak{N}_{2,t},
\end{eqnarray}
\noindent We also note that the derivations inside $\mathfrak{N}_{2,t}$ are nilpotent maps and $[id_{2,t},\widehat{\delta}]=(k-1)\widehat{\delta}$ for any $\widehat{\delta}\in \mathrm{Der}_{k}\, \mathfrak{n}_{2,t}$ ($k\geq 1$).
  
Concerning the classification problem of Lie algebras having $\mathfrak{n}_{2,t}$ as nilradical we have (cf. of \cite[Proposition 3.2]{BeCo13} and Theorem \ref{Malcev44}):

\begin{thm}\label{free nilpotent}Up to isomorphisms, $\mathfrak{g}_{2,t}=\mathrm{Der}_1^0\, \mathfrak{n}_{2,t}\oplus_{id}\mathfrak{n}_{2,t}$ is the unique faithful complex Lie algebra with radical the free nilpotent Lie algebra $\mathfrak{n}_{2,t}$.  In particular, apart from $\mathfrak{g}_{2,t}$, any nonsolvable complex Lie algebra with radical $\mathfrak{n}_{2,t}$ is a direct sum as ideals of the form either $\mathfrak{s}\oplus \mathfrak{n}_{2,t}$  or $\mathfrak{s}\oplus \mathfrak{g}_{2,t}$, where $\mathfrak{s}$ is an arbitrary semisimple Lie algebra.\hfill $\square$
 \end{thm}
 The product in $\mathfrak{g}_{2,t}$ is given by considering the product in $\mathrm{Der}_1^0\, \mathfrak{n}_{2,t}$ as subalgebra of the linear algebra of derivations (so, $[d_1, d_2]=d_1d_2-d_2d_1$), the usual product in $\mathfrak{n}_{2,t}$  and $[d,a]=d(a)$ in case $d\in \mathrm{Der}_1^0\, \mathfrak{n}_{2,t}, a\in \mathfrak{n}_{2,t}$. By using representation theory of $\mathfrak{sl}_2(\mathbb{C})$ (see \cite[section II.7]{Hu72}) we can get explicit basis for $\mathfrak{g}_{2,t}$ through the natural action of $ \mathrm{Der}_1^0\, \mathfrak{n}_{2,t}$ on $\mathfrak{m}$ and the induced action on the $j$-tensor vector space $\otimes^j \mathfrak{m}$ (for a computational approach see \cite{BeCo12}). The next result shows the way in which this method works:
\begin{prop}\label{low free nilpotent}
Up to isomorphisms, the nonsolvable complex Lie algebras with solvable radical a universal nilpotent algebra of type 2 and nilindex $t\leq 4$ are:
\begin{enumerate}

\item[a)] The trivial extensions $\mathfrak{s}\oplus \mathfrak{n}_{2,t}$ (direct sum as ideals),  $\mathfrak{s}$ any arbitrary semisimple Lie algebra.

\item[b)] The direct sum as ideals of any semisimple Lie algebra $\mathfrak{s}$ and one of the following faithful Lie algebras:
\begin{enumerate}
\item [i)] The $5$-dimensional algebra $\mathfrak{g}_{2,1}$ with basis $\{e,f,h, v_0, v_1 \}$ and nonzero products $[e,f]=h$, $[h,e]=2e$, $[h,f]=-2f$, $[h,v_0]=v_0$, $[h,v_1]=-v_1$, $[e,v_1]=v_0$ and $[f,v_0]=v_1$. In this case, $\mathfrak{n}_{2,1}=\mathbb{C} v_0\oplus, \mathbb{C}v_1$ is a $V(1)$-module of the Levi subalgebra $\mathbb{C}e\oplus \mathbb{C}f \oplus \mathbb{C}h\ (\cong  \mathrm{Der}_1^0\, \mathfrak{n}_{2,1})$.

\item[ii)] The $6$-dimensional algebra $\mathfrak{g}_{2,2}$ with basis $\{e,f,h, v_0, v_1,w_0 \}$ and non-zero products $[e,f]=h$, $[h,e]=2e$, $[h,f]=-2f$, $[h,v_0]=v_0$, $[h,v_1]=-v_1$, $[e,v_1]=v_0$, $[f,v_0]=v_1$ and $[v_0,v_1]=w_0$. In this case, $\mathfrak{n}_{2,2}=\mathbb{C} v_0\oplus \mathbb{C}v_1\oplus \mathbb{C}w_0$ is a $V(1)\oplus V(0)$-module of the Levi subalgebra $\mathbb{C}e\oplus \mathbb{C}f \oplus \mathbb{C}h\ (\cong  \mathrm{Der}_1^0\, \mathfrak{n}_{2,2})$.

\item[iii)] The $8$-dimensional algebra $\mathfrak{g}_{2,3}$ with basis $\{e,f,h, v_0, v_1,w_0, z_0, z_1 \}$ and nonzero products $[e,f]=h$, $[h,e]=2e$, $[h,f]=-2f$, $[h,v_0]=v_0$, $[h,v_1]=-v_1$, $[e,v_1]=v_0$, $[f,v_0]=v_1$,  $[h,z_0]=z_1$, $[h,z_0]=-z_0$, $[e,z_1]=z_0$, $[f,z_0]=z_1$, $[v_0,v_1]=w_0$, $[v_0,w_0]=z_0$ and $[v_1,w_0]=z_1$. In this case, $\mathfrak{n}_{2,3}=\mathbb{C} v_0\oplus \mathbb{C}v_1\oplus \mathbb{C}w_0 \oplus \mathbb{C} z_0\oplus \mathbb{C}z_1$ is a $V(1)\oplus V(0)\oplus V(1)$-module of the Levi subalgebra $\mathbb{C}e\oplus \mathbb{C}f \oplus \mathbb{C}h\ (\cong  \mathrm{Der}_1^0\, \mathfrak{n}_{2,3})$.

\item[iv)] The $11$-dimensional algebra $\mathfrak{g}_{2,4}$ with basis $\{e,f,h, v_0, v_1, w_0, z_0, z_1,$ $x_0,x_1,x_2\}$ and nonzero products $[e,f]=h$, $[h,e]=2e$, $[h,f]=-2f$, $[h,v_0]=v_0$, $[h,v_1]=-v_1$, $[e,v_1]=v_0$, $[f,v_0]=v_1$, $[h,z_0]=z_1$, $[h,z_0]=-z_1$, $[e,z_1]=z_0$, $[f,z_0]=z_1$, $[h,x_0]=2x_0$, $[h,x_2]=-2x_2$, $[e,x_1]=2x_0$, $[e,x_2]=x_1$, $[f,x_0]=x_1$, $[f,x_1]=2x_2$, $[v_0,v_1]=w_0$, $[v_0,w_0]=z_0$, $[v_1,w_0]=z_1$, $[v_0,z_0]=x_0$, $[v_0,z_1]=[v_1,z_0]=\frac{1}{2}x_1$ and $[v_1,z_1]=x_2$. In this case, $\mathfrak{n}_{2,4}=\mathbb{C} v_0\oplus \mathbb{C}v_1\oplus \mathbb{C}w_0 \oplus \mathbb{C} z_0\oplus \mathbb{C}z_1\oplus \mathbb{C}x_0 \oplus \mathbb{C} x_1\oplus \mathbb{C}x_2$ is a $V(1)\oplus V(0)\oplus V(1)\oplus V(2)$-module of the Levi subalgebra $\mathbb{C}e\oplus \mathbb{C}f \oplus \mathbb{C}h\ (\cong  \mathrm{Der}_1^0\, \mathfrak{n}_{2,4})$.
\end{enumerate}
\end{enumerate}
\end{prop}

\begin{proof} Let $L=\mathfrak{s}\oplus \mathfrak{n}_{2,t}$ be a such algebra, $\mathfrak{s}$ a Levi subalgebra. By using the adjoint (restricted) representation  $\rho=\mathrm{ad}_{\mathfrak{n}_{2,t}}$ of $L$, the radical $\mathfrak{n}_{2,t}$ is viewed as a $\mathfrak{s}$-module. In case $\rho$ is trivial, $a)$ follows; otherwise, $\rho(\mathfrak{s})$ is a semisimple subalgebra of the Levi subalgebra of $\mathrm{Der}\, \mathfrak{n}_{2,t}$. So, $\mathfrak{s}=\mathrm{Ker}\,\rho \oplus \mathfrak{s}_1$ where $\mathfrak{s}_1$ is a simple split $3$-dimensional ideal of $\mathfrak{s}$ and $\mathfrak{s}_1\oplus \mathfrak{n}_{2,t}$ is a faithful Lie algebra. From Theorem \ref{free nilpotent}, up to isomorphisms, we can assume $\mathfrak{s}_1\oplus \mathfrak{n}_{2,t}= \mathrm{Der}_1^0\, \mathfrak{n}_{2,t}\oplus_{id}\mathfrak{n}_{2,t}$. From now on, we fixed a standard basis $\{e,f,h\}$ of the split simple $3$-dimensional $ \mathrm{Der}_1^0\, \mathfrak{n}_{2,t}$ so: $[e,f]=h$, $[h,e]=2e$, $[h,f]=-2f$. Through the natural $ \mathrm{Der}_1^0\, \mathfrak{n}_{2,t}$-action, $\mathfrak{m}$ is a module of type $V(1)$ and we can also take a standard basis $\{v_0,v_1\}$ of $\mathfrak{m}$ as in \cite{Hu72}.  In case $t=1$, $b)-i)$ follows immediately. For $s\geq 2$ the homogeneous components $\mathfrak{m}^s$ of given in (\ref{grad}) are $ \mathrm{Der}_1^0\, \mathfrak{n}_{2,t}$-submodules inside the tensor product $m\otimes\mathfrak{m}^{s-1}$ (the Lie product $[\cdot,\cdot]:\mathfrak{m}\otimes\mathfrak{m}^{s-1}\to \mathfrak{m}^s$ is an onto $ \mathrm{Der}_1^0\, \mathfrak{n}_{2,t}$-module homomorphism). So, from Clebch-Gordan formula an a counting dimension argument based on (\ref{Moebius}) we have:
\begin{itemize}
\item $\mathfrak{m}^2\subseteq \mathfrak{m}\otimes\mathfrak{m}=V(1)\otimes V(1)=V(2)\oplus V(0)$; so $\mathfrak{m}^2=V(0)$, is a trivial module. Thus $\mathfrak{m}^2=\mathrm{span}\langle [v_0,v_1]\rangle$ which leads to case $b)-ii)$ for $t=2$ by defining $w_0=[v_0,v_1]$.

\item $\mathfrak{m}^3\subseteq \mathfrak{m}\otimes\mathfrak{m}^2=V(1)\otimes V(0)=V(1)$; so $\mathfrak{m}^3=V(1)$, is a $2$-irreducible module. Since $\mathfrak{m}^3=[\mathfrak{m},\mathfrak{m}^2]=\mathrm{span}\langle [v_0,w_0]=z_0,[v_1,w_0]=z_1\rangle$, $b)-iii)$ is obtained in case $t=3$.

\item $\mathfrak{m}^4\subseteq \mathfrak{m}\otimes\mathfrak{m}^3=V(1)\otimes V(1)=V(2)\oplus V(0)$; so $\mathfrak{m}^4=V(2)$, is a $3$-irreducible module. From $\mathfrak{m}^4=[\mathfrak{m},\mathfrak{m}^3]$ and $[v_1,z_0]=[v_0,z_1]$, the set $\{[v_0,z_0]=x_0, 2[v_1,z_0]=x_1,[v_1,z_0]=x_2\}$ turns out a standar basis of $V(2)$ inside $\mathfrak{m}^4$. So, $b)-iv)$ follows.  
\end{itemize}
\end{proof}

Now, from Proposition \ref{low free nilpotent}, \cite[Proposition 1.5]{Ga73} and \cite[Theorem 3.5]{BeCo13} we get the whole list of nilpotent Lie algebras of type $2$ and nilindex $\leq 4$:

\begin{cor}\label{classification}Up to isomorphisms, the nilpotent Lie algebras of type 2 and nilindex $t\leq 4$ are:
\begin{itemize}
\item[a)] $\mathfrak{n}_{2,t}$ for $t=1,2,3,4$;
\item[b)] the quotient Lie algebra $\displaystyle\frac{\mathfrak{n}_{2,3}}{I}$, where $I$ is any $1$-dimensional subspace of $\mathfrak{n}_{2,3}^3=\mathrm{span}\langle z_0, z_1\rangle$;
\item[c)] the quotient Lie algebra $\displaystyle\frac{\mathfrak{n}_{2,4}}{I}$ where $I$ is one of the following ideals:
\begin{itemize}
\item[i)] any $1$ or $2$-dimensional subspace of $\mathfrak{n}_{2,3}^4=\mathrm{span}\langle x_0, x_1,x_2\rangle$,
\item[ii)]$I=\mathrm{span}\langle z_1+\alpha x_0, x_1,x_2\rangle$, $\alpha\in \mathbb{C}$ or 
\item[iii)]$I=\mathrm{span}\langle z_0+\alpha z_1+\beta x_2, 2x_0+\alpha x_1,x_1+2\alpha x_2\rangle$, $\alpha, \beta \in \mathbb{C}$.
\end{itemize}
\end{itemize}

Moreover, the algebras in the previous list that admit a nontrivial Levi extension are those given in item a).
\end{cor}
\begin{proof}From \cite[Proposition 1.5]{Ga73}, these algebras are of the form $\displaystyle\frac{\mathfrak{n}_{2,t}}{I}$ where the ideal $I$ is contained in $\mathfrak{n}_{2,t}^2$ and such that $\mathfrak{n}_{2,t}^t\not\subseteq I$. Now,  the result for $t\leq 4$ is a straightforward computation. 
\end{proof}
\begin{rmk} Corollary \ref{classification} provides the classification of filiform Lie algebras of nilindex $\leq 4$. Apart from $\mathfrak{n}_{2,1}$ and $\mathfrak{n}_{2,2}$ the algebras in this class are given in items \emph{b)}, \emph{c)-ii)}, \emph{c)-iii)}.  By considering $2$-subspaces $I$,  item  \emph{b)-i)}  and $\mathfrak{n}_{2,3}$ gives  us all quasifiliform algebras of type 2 and nilindex $\leq 4$.
\end{rmk}

\section{Solvable extensions of $\mathfrak{n}_{2,t}$}

Now we study the solvable Lie algebras $\mathfrak{r}$ having a $2$-free nilpotent algebra $\mathfrak{n}_{2,t}$ as nilradical. The basic idea is to get $\mathfrak{r}$ by extending $\mathfrak{n}_{2,t}$ through commuting vector spaces of derivations of the Lie algebra $\mathrm{Der}_1\, \mathfrak{n}_{2,t}= \mathrm{Der}_1^0\, \mathfrak{n}_{2,t}\oplus \mathbb{C}\cdot id_{2,t}$ described in (\ref{der1}) which is isomorphic to $\mathfrak{gl}_2(\mathbb{C})$. In this section we will show that, the codimension $\mathfrak{n}_{2,t}$ on its extended solvable radical $\mathfrak{r}$ is at most $2$.


\begin{lem}\label{solvExt}Let $\mathfrak{r}$ be a complex solvable Lie algebra with nilradical $\mathfrak{n}_{2,t}$. Then, $\mathfrak{r}=\mathfrak{n}_{2,t}\oplus \mathfrak{t}$ where $\mathfrak{t}$ is a complex vector space equidimensional to the projection $\mathfrak{p}_1(\mathfrak{t})=\mathrm{proj}_{\mathrm{Der}_1\, \mathfrak{n}_{2,t}}\mathrm{ad}_{\mathfrak{n}_{2,t}}\mathfrak{t}$ which is an abelian subalgebra of $\mathrm{Der}_1\, \mathfrak{n}_{2,t}$ of dimension at most $2$ and the derivation set $\mathrm{ad}_{\mathfrak{n}_{2,t}}\mathfrak{t}$ contains no nilpotent elements. Moreover, in case $\mathfrak{t}=\mathbb{C}\cdot x+\mathbb{C}\cdot y$, $\mathfrak{p}_1(\mathfrak{t})=C_{\mathrm{Der}_1\, \mathfrak{n}_{2,t}}(\delta)$ for some $\delta \in \mathrm{Der}_1\, \mathfrak{n}_{2,t}\backslash \mathbb{C}\cdot id_{2,t}$ and the braket derivation $[\mathrm{ad}_{\mathfrak{n}_{2,t}}x,\mathrm{ad}_{\mathfrak{n}_{2,t}}y]$ is an inner derivation of $\mathfrak{n}_{2,t}$.
\end{lem}

\begin{proof} Let $\mathfrak{t}$ be any arbitrary complement of  $\mathfrak{n}_{2,t}$ in $\mathfrak{r}$. From (\ref{R2inN}), we have the relation $\mathfrak{t}^2\subseteq \mathfrak{r}^2\subseteq\mathfrak{n}_{2,t}$. Note that, since the nilradical of $\mathfrak{m}$ is $\mathfrak{n}_{2,t}$, the vector linear space $ad_{\mathfrak{n}_{2,t}}\mathfrak{t}$ is a subspace of derivations in $\mathfrak{n}_{2,t}$ with  no nilpotent linear maps. Consider now the adjoint representation $\mathrm{ad}_{\mathfrak{n}_{2,t}}:\mathfrak{r}\to \mathrm{Der}\,\mathfrak{n}_{2,t}$ and the projection homomorphism $\pi_0: \mathrm{Der}\,\mathfrak{n}_{2,t}\to \mathrm{Der}_1\,\mathfrak{n}_{2,t}$. Then, $\mathrm{Ker}\, \pi_0\circ \mathrm{ad}_{\mathfrak{n}_{2,t}}=\{a\in \mathfrak{r} : \mathrm{ad}_{\mathfrak{n}_{2,t}}\,a \in \mathfrak{N}_{2,t}\}=\mathfrak{n}_{2,t}$. Thus the quotient $\displaystyle\frac{\mathfrak{r}}{\mathfrak{n}_{2,t}}$ is an abelian Lie algebra equidimensional to $\mathfrak{t}$ and isomorphic to $\mathrm{Im}\, \pi_0\circ \mathrm{ad}_{\mathfrak{n}_{2,t}}$. Since $\mathrm{Der}_1\,\mathfrak{n}_{2,t}\cong \mathfrak{gl}_2(\mathbb{C})$, up to isomorphisms, $\mathrm{Im}\, \pi_0\circ \mathrm{ad}_{\mathfrak{n}_{2,t}}$ is a set of commutative linear transformations over a $2$-dimensional vector space. For any $\alpha\cdot  \mathrm{id}_\mathfrak{m}\neq f\in \mathrm{Im}\, \pi_0\circ \mathrm{ad}_{\mathfrak{n}_{2,t}}$, the centralizer of $f$, $C_{\mathfrak{gl}_2(\mathbb{C})}(f)$ is a $2$-dimensional subspace which contains $\mathrm{Im}\, \pi_0\circ \mathrm{ad}_{\mathfrak{n}_{2,t}}$. Note that for $2$-dimensional extensions, $\mathrm{Im}\, \pi_0\circ \mathrm{ad}_{\mathfrak{n}_{2,t}}$ is just a centralizer. The final assertion follows from the fact $\mathrm{ad}_{\mathfrak{n}_{2,t}}[x,y]=[\mathrm{ad}_{\mathfrak{n}_{2,t}}x,\mathrm{ad}_{\mathfrak{n}_{2,t}}y]$ and $[x,y] \in \mathfrak{n}_{2,t}$.
\end{proof}

\subsection{Derivations and automorphisms}

From Proposition \ref{DerivationN2}, the set of derivations of any Lie algebra $\mathfrak{n}_{2,t}$ is completely determined by $Hom(\mathfrak{m},\mathfrak{m})$ and $Hom(\mathfrak{m},\mathfrak{n}_{d,t}^2)=\oplus_{2\leq j\leq t}Hom(\mathfrak{m},\mathfrak{m}^j)$. Moreover, any automorphism appears as extension of a map of the general linear group $GL(\mathfrak{m})$ according to \cite[Proposition 3]{Sa70}. Along this section, the derivation algebras and automorphisms groups of free nilpotent Lie algebras will be represented through matrices $(\alpha_{ij})$ relative to the basis given in Proposition \ref{low free nilpotent} ($\mathfrak{m}=\mathbb{C}\cdot v_0\oplus \mathbb{C}\cdot v_1$, only nonzero products are displayed): 

\begin{itemize}

\item [$\mathfrak{n}_{2,1}:$] Abelian $2$-dimensional. 
\item []$\mathrm{Der}\, \mathfrak{n}_{2,1}=\mathfrak{gl}_2(\mathbb{C})=\mathrm{Der}_1\mathfrak{n}_{2,1}:$$$D_{\mathbf{u}}^\beta=\left(
\begin{array}{cc}
\alpha_1+\beta& \alpha_2 \\
 \alpha_3& -\alpha_1+\beta \\

\end{array}\right),\  {\mathbf{u}}=(\alpha_1,\alpha_2,\alpha_3);$$
\item[]extended derivation of the identity map $I_2$ is $I_{2,1}=I_2$;
\item []$\mathfrak{N}_{2,1}=\mathrm{Inner}\, \mathfrak{n}_{2,1}=0$;
\item []$\mathrm{Aut}\, \mathfrak{n}_{2,1}=GL_2(\mathbb{C}):$$$\Phi_{\mathbf{v}}=\left(
\begin{array}{cc}
\alpha_1& \alpha_2 \\
 \alpha_3& \alpha_4 \\

\end{array}\right), \ {\mathbf{v}}=(\alpha_1,\dots,\alpha_4) \ \mathrm{and}\ \epsilon=\alpha_1\alpha_4-\alpha_2\alpha_3\neq 0.$$
\item[]
\item [$\mathfrak{n}_{2,2}:$] Heisenberg $3$-dimensional, $[v_0,v_1]=w_0.$
\item []$\mathrm{Der}\, \mathfrak{n}_{2,2}:$ $D_{\mathbf{u}}^\beta=\left(
\begin{array}{ccc}
\alpha_1+\beta& \alpha_2 & 0  \\
 \alpha_3& -\alpha_1+\beta& 0 \\
\alpha_4& \alpha_5& 2\beta \\
\end{array}\right)$, \ ${\mathbf{u}}=(\alpha_1,\dots,\alpha_5)$;
\item[]
\item []$\mathrm{Der}_1\, \mathfrak{n}_{2,2}:$ $D_{(\alpha_1,\alpha_2,\alpha_3,0,0)}^\beta=\left(
\begin{array}{ccc}
\alpha_1+\beta& \alpha_2 & 0  \\
 \alpha_3& -\alpha_1+\beta& 0 \\
0& 0& 2\beta \\
\end{array}\right);$
\item[]extended derivation of $I_2$ is the $3\times 3$ matrix $I_{2,2}=D_\mathbf{0}^1;$
\item[]
\item []$\mathfrak{N}_{2,2}=\mathrm{Inner}\, \mathfrak{n}_{2,2}:$ $D_{(0,0,0,\alpha_4,\alpha_5)}^0=\left(
\begin{array}{ccc}
0& 0 & 0  \\
 0 & 0& 0  \\
 \alpha_4& \alpha_5 &  0 \\
\end{array}\right).$
\item[]
\item []$\mathrm{Aut}\, \mathfrak{n}_{2,2}:$ $ \Phi_{\mathbf{v}}=\left(
\begin{array}{ccc}
 \alpha_1& \alpha_2 & 0  \\
 \alpha_3& \alpha_4& 0 \\
\alpha_5& \alpha_6&   \epsilon \\
\end{array}\right), {\mathbf{v}}=(\alpha_1,\dots,\alpha_5).$
\item[]
\item [$\mathfrak{n}_{2,3}:$]denoted as $\mathcal{L}_{3,5}$ in \cite{An11}, $[v_0,v_1]=w_0, [v_i,w_0]=z_i, i=0,1.$
\item []$\mathrm{Der}\, \mathfrak{n}_{2,3}:${\small $$D_{\mathbf{u}}^ \beta=\left(
\begin{array}{ccccc}
 \alpha_1+\beta& \alpha_2 & 0 & 0 & 0 \\
 \alpha_3& -\alpha_1+\beta& 0 & 0 & 0 \\
\alpha_4& \alpha_5& 2\beta& 0 & 0 \\
 \alpha_6& \alpha_7 & \alpha_5 & \alpha_1+3\beta & \alpha_2\\
 \alpha_8& \alpha_9 & -\alpha_4& \alpha_3 & -\alpha_1+3\beta \\
\end{array}\right), {\mathbf{u}}=(\alpha_1,\dots,\alpha_9);$$}
\item []$\mathrm{Der}_1\mathfrak{n}_{2,3}:${\small $$D_{(\alpha_1,\alpha_2,\alpha_3,0,0,0,0,0,0)}^ \beta=\left(
\begin{array}{ccccc}
 \alpha_1+\beta& \alpha_2 & 0 & 0 & 0 \\
 \alpha_3& -\alpha_1+\beta& 0 & 0 & 0 \\
0& 0& 2\beta& 0 & 0 \\
0& 0& 0 & \alpha_1+3\beta & \alpha_2\\
0& 0 & 0& \alpha_3 & -\alpha_1+3\beta \\
\end{array}\right);$$}
\item[]extended derivation of $I_2$ is the $5\times 5$ matrix $I_{2,3}=D_\mathbf{0}^1;$
\item []$\mathfrak{N}_{2,3}:$ $D_{(0,0,0,\alpha_4,\dots,\alpha_9)}^ 0=\left(
\begin{array}{ccccc}
0& 0 & 0 & 0 & 0 \\
0& 0& 0 & 0 & 0 \\
\alpha_4& \alpha_5& 0& 0 & 0 \\
 \alpha_6& \alpha_7 & \alpha_5 &0&0\\
 \alpha_8& \alpha_9 & -\alpha_4& 0& 0 \\
\end{array}\right);$
\item []$\mathrm{Inner}\, \mathfrak{n}_{2,3}:$ $D_{(0,0,0,\alpha_4,\alpha_5, \alpha_6,0,0,0)}^0=\left(
\begin{array}{ccccc}
0& 0 & 0&0&0  \\
 0 & 0& 0&0&0  \\
 \alpha_4& \alpha_5&  0&0&0 \\
 \alpha_6& 0 & \alpha_5&0 &0  \\
 0& \alpha_6 & -\alpha_4& 0&0  \\
\end{array}\right);$

\item []$\mathrm{Aut}\, \mathfrak{n}_{2,3}:$ ${\small \Phi_{\mathbf{v}}=\left(
\begin{array}{ccccc}
 \alpha_1& \alpha_2 & 0 &0&0  \\
 \alpha_3& \alpha_4 & 0 &0&0  \\
 \alpha_5& \alpha_6 &  \epsilon &0&0 \\
\alpha_7& \alpha_8&\alpha_1\alpha_6-\alpha_2\alpha_5 &\epsilon \alpha_1 &\epsilon \alpha_2  \\
\alpha_9& \alpha_{10} & \alpha_3\alpha_6 -\alpha_4\alpha_5&\epsilon \alpha_3&\epsilon \alpha_4  \\

\end{array}\right), {\mathbf{v}}=(\alpha_1,\dots,\alpha_{10})}$.



\end{itemize}

\subsection{Solvable $1$-extensions of $\mathfrak{n}_{2,t}$ for $t=1,2,3$}
\bigskip

Solvable $1$-extensions $\mathfrak{t}=k\cdot x$ of $\mathfrak{n}_{2,t}$ follows from no nilpotent derivations $\mathrm{ad}_{\mathfrak{n}_{2,t}} x=D+O$ where $D\in \mathrm{Der}_1\mathfrak{n}_{2,t}$ and $O\in \mathfrak{N}_{2,t}$. Moreover, the isomorphism problem between two different one dimensional extensions is solved in an easy way:  $\mathfrak{n}_{2,t}\oplus k\cdot x \cong \mathfrak{n}_{2,t}\oplus k\cdot x'$ if and only if there exists a $\Phi \in \mathrm{Aut}\, \mathfrak{n}_{2,t}$ such that $\Phi \cdot(D+O)\cdot \Phi^{-1}=\mathrm{ad}\, a+\alpha\cdot (D'+O')$ for some $a\in \mathfrak{n}_{2,t}$, $\alpha\in \mathbb{C}$. Since $\Phi \cdot\mathrm{Inner}\, \mathfrak{n}_{2,t}\cdot \Phi^{-1}=\mathrm{Inner}\, \mathfrak{n}_{2,t}$, we can consider w.l.o.g. extensions by derivations with zero projection on $\mathrm{Inner}\, \mathfrak{n}_{2,t}$.

\subsubsection{ $1$-extensions in case $\mathfrak{n}_{2,1}:$ }

\noindent From $\mathrm{Inner}\, \mathfrak{n}_{2,1}=0$, we have that solvable $1$-extensions are given by Jordan forms of $2\times 2$ matrices. So, up to isomorphisms, we can extend by one of the following derivations:

\begin{enumerate}
\item  $D_{(0,0,1)}^1=\left(\begin{array}{cc}
 1 & 0   \\
 1 & 1  \\
 \end{array}
\right).$
\item $D_{(\frac{1-\alpha}{2},0,0)}^{\frac{1+\alpha}{2}}=\left(\begin{array}{cc}
 1 & 0   \\
 0 & \alpha  \\
 \end{array}
\right)$, $\alpha\in \mathbb{C};$
\end{enumerate} 

\subsubsection{ $1$-extensions in case $\mathfrak{n}_{2,2}:$} 

\noindent Since $\mathrm{Inner}\,\mathfrak{n}_{2,2}=\mathfrak{N}_{2,2}$, solvable $1$-extensions are also given by Jordan forms of $2\times 2$ matrices. As in the previous case, we have two possibilities:
\begin{enumerate}
\item $D_{(0,0,1,0,0)}^1=\left(
\begin{array}{ccc}
 1 & 0 & 0  \\
 1 & 1 & 0  \\
 0&0& 2 \\
\end{array}
\right).$
\item $D_{(\frac{1-\alpha}{2},0,0,0,0)}^{\frac{1+\alpha}{2}}=\left(
\begin{array}{ccc}
 1 & 0 & 0  \\
 0 & \alpha & 0  \\
0 &0 & 1+\alpha \\
\end{array}
\right)$, $\alpha\in \mathbb{C};$

\end{enumerate}

\subsubsection{ $1$-extensions in case $\mathfrak{n}_{2,3}:$ }
\smallskip

\noindent Note that derivations with no projection on $\mathrm{Inner}\, \mathfrak{n}_{2,3}$ are of the form:{\small$$\left(
\begin{array}{ccc}
A & \mathbf{0}_{2,1} & \mathbf{0}_{2,2}  \\
\mathbf{0}_{1,2}& \mathrm{tr}\, A & \mathbf{0}_{1,2}  \\
 M&\mathbf{0}_{2,1}& A+2\beta I_2 \\
\end{array}
\right)=\left(
\begin{array}{ccc}
A & \mathbf{0}_{2,1} & \mathbf{0}_{2,2}  \\
\mathbf{0}_{1,2}& \mathrm{tr}\, A & \mathbf{0}_{1,2}  \\
\mathbf{0}_{2,2}&\mathbf{0}_{2,1}& A+2\beta I_2 \\
\end{array}
\right)+\left(
\begin{array}{ccc}
 \mathbf{0}_{2,2}&  \mathbf{0}_{2,1} &  \mathbf{0}_{2,2}  \\
 \mathbf{0}_{1,2}& \mathbf{0}_{1,1} &  \mathbf{0}_{1,2}  \\
 M& \mathbf{0}_{2,1}& \mathbf{0}_{2,2} \\
\end{array}
\right),$$}$\mathbf{0}_{m,n}$ zero matrix $m\times n$, $A$ and $M={\small \left(
\begin{array}{cc}
\alpha_6 & \alpha_7  \\
 \alpha_8& 0  \\
\end{array}
\right)}$ $2\times 2$ matrices. According to (\ref{gradder}), the second summand in previous matrix decomposition corresponds to a derivation $\delta: \mathfrak{n}_{2,3}\to Z(\mathfrak{n}_{2,3})=\mathfrak{n}_{2,3}^3$ which belongs to $\mathrm{Der}_3\mathfrak{n}_{2,3}$, so $\delta$ has the matrix general form:$$\left(
\begin{array}{ccc}
\mathbf{0}_{2,2} & \mathbf{0}_{2,1} & \mathbf{0}_{2,2}  \\
\mathbf{0}_{1,2}& \mathbf{0}_{1,1} & \mathbf{0}_{1,2}  \\
 {\small \begin{array}{cc}
 \alpha_6 & \alpha_7   \\
 \alpha_8 & \alpha_9  \\
 \end{array}}
&\mathbf{0}_{2,1}& \mathbf{0}_{2,2}\\
\end{array}
\right)
$$We note that, $\mathrm{Der}_3\mathfrak{n}_{2,3}$ is an abelian subalgebra which is invariant by conjugation. This subalgebra  contains a $1$-dimensional subspace of inner derivations. Then, up to isomorphisms, the extensions we are looking for are given by one of the following type of derivations:

\begin{enumerate}
\item[(a)] The derivation matrix $D_{\mathbf{u}}^1$ where $\mathbf{u}=(0,0,1,0,0,\alpha_6,\alpha_7,\alpha_8,0)$, i.e.:$$D_{\mathbf{u}}^1=\left(
\begin{array}{ccccc}
 1 & 0 & 0 & 0 & 0 \\
 1 & 1 & 0 & 0 & 0 \\
 0 & 0 & 2 & 0 & 0 \\
 \alpha_6 & \alpha_7 & 0 & 3 & 0 \\
 \alpha_8 & 0 & 0 & 1 & 3 \\
\end{array}
\right).$$Denoting ${\mathbf{v}}=(1,0,0,1,0,0,\frac{1}{4} (2\alpha_6+\alpha_7), \frac{\alpha_7}{2},  -\frac{\alpha_6}{4}-\frac{\alpha_7}{4}+\frac{\alpha_8}{2},-\frac{\alpha_7}{4})$, the automorphism $\Phi_{\mathbf{v}}$ transforms previous derivation by conjugation into:$$\Phi_{\mathbf{v}}^{-1}\cdot D_{\mathbf{u}}^1\cdot \Phi_{\mathbf{v}}=D_{\mathbf{u}_1}^1$$ where $\mathbf{u}_1=(0,0,1,0,0,0,0,0,0)$. Then, we get:$$D_{\mathbf{u}_1}^1=\left(
\begin{array}{ccccc}
 1 & 0 & 0 & 0 & 0 \\
 1 & 1 & 0 & 0 & 0 \\
 0 & 0 & 2 & 0 & 0 \\
 0 & 0 & 0 & 3 & 0 \\
 0 & 0 & 0 & 1 & 3 \\
\end{array}
\right).$$
\item[(b)] The derivation matrix $D_\mathbf{u}^{\frac{1+\alpha}{2}}$ with $\mathbf{u}=(\frac{1-\alpha}{2},0,0,0,0,\alpha_6,\alpha_7,\alpha_8,0)$, i.e.:$$D_\mathbf{u}^{\frac{1+\alpha}{2}}=\left(
\begin{array}{ccccc}
 1 & 0 & 0 & 0 & 0 \\
  0& \alpha & 0 & 0 & 0 \\
 0 & 0 & 1+\alpha & 0 & 0 \\
 \alpha_6 & \alpha_7 & 0 & 2+\alpha & 0 \\
 \alpha_8 & 0 & 0 & 0& 1+2\alpha \\
\end{array}
\right).$$Now  three subcases can be considered:
\smallskip

\begin{enumerate}
\item[](b.1) $\alpha=-1$: By conjugation through the automorphism $\Phi_\mathbf{v}$, with $\mathbf{v}=(1,0,0,1,0,0,0, \frac{\alpha_7}{2}, \frac{\alpha_8}{2},0)$, the derivation $D_\mathbf{u}^{\frac{1+\alpha}{2}}=D_\mathbf{u}^0$ becomes into:$$\Phi_\mathbf{v}^{-1}\cdot D_\mathbf{u}^0\cdot \Phi_\mathbf{v}=D_{\mathbf{w}}^0,$$ where $\mathbf{w}=(1,0,0,0,0,\alpha_6,0,0,0)$. If $\alpha_6\neq 0$, by using $\Phi_{\mathbf{v'}}$ with $\mathbf{v'}=(\frac{1}{\sqrt{\alpha_6}},0,0,\frac{1}{\sqrt{\alpha_6}},0,0,0, 0, 0,0)$, we get:$$\Phi_{\mathbf{v'}}^{-1}\cdot D_{\mathbf{w}}^0\cdot \Phi_{\mathbf{v'}}=D_{\mathbf{u}_3}^0,$$$\mathbf{u}_3=(1,0,0,0,0,1,0,0,0)$. So, up to isomorphisms, two new possible derivations appear ($\mathbf{u}_2=(1,0,0,0,0,0,0,0,0)$ is related to $\alpha_6=0$):$${\small
D_{\mathbf{u}_2}^0=\left(
\begin{array}{ccccc}
 1 & 0 & 0 & 0 & 0 \\
 0 & -1 & 0 & 0 & 0 \\
 0 & 0 & 0 & 0 & 0 \\
 0 & 0 & 0 & 1 & 0 \\
 0 & 0 & 0 & 0 & -1 \\
\end{array}
\right)\ \mathrm{and}\ 
D_{\mathbf{u}_3}^0=\left(
\begin{array}{ccccc}
 1 & 0 & 0 & 0 & 0 \\
 0 & -1 & 0 & 0 & 0 \\
 0 & 0 & 0 & 0 & 0 \\
 1 & 0 & 0 & 1 & 0 \\
 0 & 0 & 0 & 0 & -1 \\
\end{array}
\right).}
$$
\item[](b.2) $\alpha=0$: Using $\mathbf{v}=(1,0,0,1,0,0,\alpha_6, \frac{\alpha_7}{2},0,0)$, the automorphism $\Phi_\mathbf{v}$, transforms $D_\mathbf{u}^{\frac{1+\alpha}{2}}=D_\mathbf{u}^{\frac{1}{2}}$ into:$$\Phi_\mathbf{v}^{-1}\cdot D_\mathbf{u}^{\frac{1}{2}}\cdot \Phi_\mathbf{v}=D_\mathbf{w}^{\frac{1}{2}},$$ where $\mathbf{w}=(\frac{1}{2},0,0,0,0,0,0,\alpha_8,0)$. Now, in case $\alpha_8\neq 0$, taking $\mathbf{v'}=(\frac{1}{\sqrt{\alpha_8}},0,0,\frac{1}{\sqrt{\alpha_8}},0,0,0, 0, 0,0)$ we have:$$\Phi_\mathbf{v'}^{-1}\cdot D_\mathbf{w}^{\frac{1}{2}}\cdot \Phi_\mathbf{v'}=D_{\mathbf{u}_5}^{\frac{1}{2}},$$$\mathbf{u}_5=(\frac{1}{2},0,0,0,0,0,0,1,0)$. So, up to isomorphisms, two additional derivations appear ($\mathbf{u}_4=(\frac{1}{2},0,0,0,0,0,0,0,0)$ is related to $\alpha_8=0$):
$${\small
D_{\mathbf{u}_4}^{\frac{1}{2}}=\left(
\begin{array}{ccccc}
 1 & 0 & 0 & 0 & 0 \\
 0 & 0 & 0 & 0 & 0 \\
 0 & 0 & 1 & 0 & 0 \\
 0 & 0 & 0 & 2 & 0 \\
 0 & 0 & 0 & 0 & 1 \\
\end{array}
\right)\ \mathrm{and}\ 
D_{\mathbf{u}_5}^{\frac{1}{2}}=\left(
\begin{array}{ccccc}
 1 & 0 & 0 & 0 & 0 \\
 0 & 0 & 0 & 0 & 0 \\
 0 & 0 & 1 & 0 & 0 \\
 0 & 0 & 0 & 2 & 0 \\
 1 & 0 & 0 & 0 & 1 \\
\end{array}
\right).}
$$

\item[](b.3) $\alpha\neq 0,-1$: Taking $\mathbf{v}=(1,0,0,1,0,0,\frac{\alpha_6}{1+\alpha}, \frac{\alpha_7}{2}, \frac{\alpha_8}{2\alpha},0)$, we have:$$\Phi_\mathbf{v}^{-1}\cdot D_\mathbf{u}^{\frac{1+\alpha}{2}}\cdot \Phi_\mathbf{v}=D_{\mathbf{u}_6}^{\frac{1+\alpha}{2}},$$$\mathbf{u}_6=(\frac{1-\alpha}{2},0,0,0,0,0,0,0,0)$. So, up to isomorphisms, we get the derivation:$$D_{\mathbf{u}_6}^{\frac{1+\alpha}{2}}=\left(
\begin{array}{ccccc}
 1 & 0 & 0 & 0 & 0 \\
 0 & \alpha & 0 & 0 & 0 \\
 0 & 0 & 1+\alpha & 0 & 0 \\
 0 & 0 & 0 & 2+\alpha & 0 \\
 0 & 0 & 0 & 0 & 1+2\alpha \\
\end{array}
\right).$$
\end{enumerate}

\end{enumerate}
\emph{Summarizing:} $1$-extensions of $\mathfrak{n}_{2,3}$ are given by one of the following derivations:
\begin{enumerate}
\item $D_{\mathbf{u}_1}^1$ where $\mathbf{u}_1=(0,0,1,0,0,0,0,0,0)$;
\item $D_{\mathbf{u}_3}^0$, where $\mathbf{u}_3=(1,0,0,0,0,1,0,0,0)$;
\item $D_{\mathbf{u}_5}^\frac{1}{2}$, where $\mathbf{u}_5=(\frac{1}{2},0,0,0,0,0,0,1,0)$;
\item $D_{\mathbf{u}_6}^\frac{1+\alpha}{2}$, where $\mathbf{u}_6=(\frac{1-\alpha}{2},0,0,0,0,0,0,0,0)$ and $\alpha \in \mathbb{C}$.
\end{enumerate}

\subsection{Solvable $2$-extensions}
\bigskip

Solvable $2$-extensions $\mathfrak{t}=\mathbb{C}\cdot x\oplus \mathbb{C}\cdot y$ are related to $2$-dimensional centralizer subalgebras:
\begin{equation}\label{central}C_{{Der}_1\mathfrak{n}_{2,t}}(D)=\mathbb{C}\cdot I_{2,t}\oplus \mathbb{C}\cdot D,
\end{equation}inside $\mathrm{Der}_1\mathfrak{n}_{2,t}$. According to Lemma \ref{solvExt}, $\mathrm{ad}_{\mathfrak{n}_{2,t}}x=I_{2,t}+O_1$ and $\mathrm{ad}_{\mathfrak{n}_{2,t}}y=D+O_2$ where $O_i \in \mathfrak{N}_{2,t}$; in fact, we can assume $O_i$ have no projection on $\mathrm{Inner}\, \mathfrak{n}_{2,t}$. Then, we can get a $2$-extension in case the set $\mathrm{span}\langle I_{2,t}+O_1,D+O_2\rangle$ has no nilpotent derivations and:
\begin{equation}\label{innercondition}[I_{2,t}+O_1,D+O_2]=[I_{2,t},O_2]+[O_1,D]+[O_1,O_2]\in  \mathrm{Inner}\, \mathfrak{n}_{2,t}. 
\end{equation}

\subsubsection{ $2$-extensions in case $\mathfrak{n}_{2,1}:$ }
\smallskip

\noindent  Since $\mathfrak{N}_{2,1}=0$, solvable extensions are just given by $\mathcal{C}=C_{{Der}_1\mathfrak{n}_{2,1}}(D)$, where $D$ is a  $2\times 2$ matrix ($I_{2,1}=I_2$). If $D$ is a non-semisimple (non nilpotent) matrix, up to isomorphisms and rescaling if necessary, we can assume that  $D=D_{(0,0,1)}^1$; but then, $D-I_2\in \mathcal{C}$ is a nilpotent derivation, which is not possible. Hence, $D$ is a semisimple matrix with two different eigenvalues, so we can assume w.l.o.g. $D=D_{(\frac{1-\alpha}{2},0,0)}^{\frac{1+\alpha}{2}}$, $\alpha\neq 1$ and therefore:$$C_{\mathfrak{n}_{2,1}}(D)=\mathbb{C}\cdot I_2 \oplus\mathbb{C}\cdot \left(\begin{array}{cc}
 1 & 0   \\
 0 & -1  \\
 \end{array}
\right).$$Now, the extension $\mathfrak{t}$ is given by the derivations $\mathrm{ad}_{\mathfrak{n}_{2,1}}x=I_{2,1}=I_2$ and $\mathrm{ad}_{\mathfrak{n}_{2,1}}y=D_{(1,0,0)}^0$. Note that, $\mathfrak{t}^2=\mathbb{C}\cdot [x,y]\subseteq Z(\mathfrak{n}_{2,1})=\mathfrak{n}_{2,1}$. Thus, we can consider the (unique) extension $\mathfrak{n}_{2,1}\oplus \mathfrak{t}_1$ with $\mathfrak{t}_1=\mathbb{C}\cdot x \oplus \mathbb{C}\cdot(y-[x,y])$ an abelian subalgebra ($x$ acts as the identity on $\mathfrak{n}_{2,1}$).

\subsubsection{ $2$-extensions in case $\mathfrak{n}_{2,2}:\quad$ }
\smallskip

\noindent In this case, $I_{2,2}=D_\mathbf{0}^1$ and $\mathfrak{N}_{2,2}=\mathrm{Inner}\, \mathfrak{n}_{2,2}$. From the last condition, solvable extensions are also given by $\mathcal{C}=C_{{Der}_1\mathfrak{n}_{2,2}}(D)$, $D=D_{(\alpha_1,\alpha_2,\alpha_3,0,0)}^\beta$. As in the previous case, $D$ must be semisimple. Otherwise, we can assume $D=D_{(0,0,1,0,0)}^1$, but then we get that $D-I_{2,2}\in \mathcal{C}$ is a nilpotent derivation. Hence, up to isomorphisms,  we can assume $D=D_{(\frac{1-\alpha}{2},0,0,0,0)}^{\frac{1+\alpha}{2}}$ with $\alpha\neq 1$ and therefore:$${\mathfrak{n}_{2,1}}(D)=\mathbb{C}\cdot I_{2,2} \oplus\mathbb{C}\cdot \left(\begin{array}{ccc}
 1 & 0 &0   \\
 0 & -1 &0  \\
  0 & 0&0  \\
 \end{array}
\right)
$$Now, the extension $\mathfrak{t}$ is given by the derivations $\mathrm{ad}_{\mathfrak{n}_{2,2}}x=I_{2,2}$ and $\mathrm{ad}_{\mathfrak{n}_{2,2}}y=D_{(1,0,0,0,0)}^0$. Since the extension follows from an abelian subalgebra of derivations, we must observe that in the extended algebra $\mathfrak{n}_{2,2}\oplus \mathfrak{t}$, $[x,y]\in Z(\mathfrak{n}_{2,2})=\mathbb{C}\cdot w_0$. Then, we can consider that the (unique) extension is given by the abelian subalgebra $\mathfrak{t}_1=\mathbb{C}\cdot x \oplus \mathbb{C}\cdot(y-\frac{1}{2}[x,y])$ (in this case, $x$ acts as $2\cdot id$ on $\mathbb{C}\cdot w_0$.)

\subsubsection{ $2$-extensions in case $\mathfrak{n}_{2,3}:\quad$ }
\smallskip

\noindent We look for  centralizer subalgebras $\mathcal{C}=C_{Der_1\mathfrak{n}_{2,3}}(D)$, $D=D_\mathbf{u}^\beta$ where $\mathbf{u}=(\alpha_1,\alpha_2,\alpha_3,0,0,0,0,0,0)$. Since derivations inside $\mathfrak{N}_{2,3}$ without projection on $\mathrm{Inner}\, \mathfrak{n}_{2,3}$ belongs to  $Der_3\mathfrak{n}_{2,3}$, which is a subspace invariant by conjugation through automorphisms,  and taking in account that $\mathrm{span}\langle I_{2,t}+O_1,D+O_2\rangle$ has no nilpotent derivations, in a similar vein as in previous cases, we get that $I_{2,t}+O_1=D_\mathbf{u}^1$ and $D+O_2=D_\mathbf{w}^1$ where $\mathbf{u}=(0,0,0,0,0,\alpha_6,\alpha_7,\alpha_8,0)$ and $\mathbf{w}=(1,0,0,0,0,\beta_6,\beta_7,\beta_8,0;0)$, i.e.:$$
D_\mathbf{u}^1=\left(
\begin{array}{ccccc}
 1 & 0 & 0 & 0 & 0 \\
 0 & 1 & 0 & 0 & 0 \\
 0 & 0 & 2 & 0 & 0 \\
 \alpha_6 & \alpha_7  & 0 & 3 & 0 \\
 \alpha_8 & 0 & 0 & 0 & 3 \\
\end{array}
\right)\ \mathrm{and}\ 
D_\mathbf{w}^1=\left(
\begin{array}{ccccc}
 1 & 0 & 0 & 0 & 0 \\
 0 & -1 & 0 & 0 & 0 \\
 0 & 0 & 0 & 0 & 0 \\
 \beta_6 & \beta_7 & 0 & 1 & 0 \\
 \beta_8 & 0  & 0 & 0 & -1 \\
\end{array}
\right).
$$Now,$$[D_\mathbf{u}^1,D_\mathbf{w}^1]=\left(
\begin{array}{ccccc}
 0 & 0 & 0 & 0 & 0 \\
 0 & 0 & 0 & 0 & 0 \\
 0 & 0 & 0 & 0 & 0 \\
 2\beta_6 & 2(\beta_7-\alpha_7)  & 0 & 0 & 0 \\
 2(\alpha_8+\beta_8)  & 0 & 0 & 0 & 0 \\
\end{array}
\right)
$$is an inner derivation if and only if $\beta_6=0$, $\beta_7=\alpha_7$ and $\beta_8=-\alpha_8$. Considering the automorphism: $\Phi_{\mathbf{v}}$ where $\mathbf{v}=(1,0,0,1,0,0,-\frac{\alpha_6}{2}, -\frac{\alpha_7}{2}, -\frac{\alpha_8}{2},0)$, by conjugation, we get the derivations:$$\Phi_\mathbf{v}^{-1}\cdot D_\mathbf{u}^1\cdot \Phi_\mathbf{v}=\left(
\begin{array}{ccccc}
 1 & 0 & 0 & 0 & 0 \\
 0 & 1 & 0 & 0 & 0 \\
 0 & 0 & 2 & 0 & 0 \\
 0 & 0  & 0 & 3 & 0 \\
0 & 0 & 0 & 0 & 3 \\
\end{array}
\right)=I_{2,3}\ \mathrm{and}$$
$$\Phi_\mathbf{v}^{-1}\cdot D_\mathbf{w}^1\cdot \Phi_\mathbf{v}=\left(
\begin{array}{ccccc}
 1 & 0 & 0 & 0 & 0 \\
 0 & -1 & 0 & 0 & 0 \\
 0 & 0 & 0 & 0 & 0 \\
 0 & 0 & 0 & 1 & 0 \\
 0 & 0  & 0 & 0 & -1 \\
\end{array}
\right)=D_{\mathbf{u}_1}^0,$$where $\mathbf{u}_1=(1,0,0,0,0,0,0,0,0)$. Hence, up to isomorphisms, there is a unique way to extend through derivations which is given by declaring  $\mathrm{ad}_{\mathfrak{n}_{2,2}}x=I_{2,3}$ and $\mathrm{ad}_{\mathfrak{n}_{2,2}}y=D_{\mathbf{u}_1}^0$. Finally, we must observe that in the corresponding extension $\mathfrak{n}_{2,3}\oplus \mathfrak{t}$, $[x,y]\in Z(\mathfrak{n}_{2,3})=\mathbb{C}\cdot z_0\oplus \mathbb{C}\cdot z_1$. Since $x$ acts as $3\cdot id$ on $Z(\mathfrak{n}_{2,3})$, we can consider the (unique) extension $\mathfrak{n}_{2,3}\oplus \mathfrak{t}_1$ where $\mathfrak{t}_1=\mathrm{span}\langle x, y-\frac{1}{3}[x,y]\rangle$ is an abelian subalgebra.

\section{Non-solvable extensions of $\mathfrak{n}_{2,t}$}

From Theorem \ref{free nilpotent} and  \cite[Theorem 2.1 and 2.2]{Tu92} (see also \cite[Proposition 2.2]{BeCo13}), the faithful nonsolvable Lie algebras with nilradical $\mathfrak{n}_{2,t}$ are of the form, \begin{equation}\label{nonsolvex}
\mathfrak{g}=\mathfrak{g}_{2,t}\oplus \mathfrak{t}_0
\end{equation}
where $\mathfrak{g}_{2,t}=\mathrm{Der}_1^0\,\mathfrak{n}_{2,t}\oplus_{id}\mathfrak{n}_{2,t}$ is as described in Theorem \ref{free nilpotent}, $\mathfrak{r}=\mathfrak{n}_{2,t}\oplus \mathfrak{t}_0$ is the solvable radical of $\mathfrak{g}$ and $\mathfrak{t}_0$ is a $\mathrm{Der}_1^0\,\mathfrak{n}_{2,t}$-trivial module with no ad-nilpotent elements. So, in the Lie algebra $\mathfrak{g}$, we have $[\mathrm{Der}_1^0\,\mathfrak{n}_{2,t},\mathfrak{t}_0]=0$. Denote by $(\mathfrak{n}_{2,t})_0$ the sum of all trivial modules in $\mathfrak{n}_{2,t}$ and note that this vector space is a subalgebra of $\mathfrak{n}_{2,t}$. Since $\mathfrak{t}_0^2 \subseteq (\mathfrak{n}_{2,t})_0$, the direct sum:
 \begin{equation}\label{zeromod}
\mathfrak{r}_0=\mathfrak{t}_0 \oplus(\mathfrak{n}_{2,t})_0
\end{equation}
\noindent is a subalgebra of $\mathfrak{r}$ that contains the (solvable) subalgebra generated through $\mathfrak{t}_0$. But from the Jacobi identity in the Lie algebra (\ref{nonsolvex}),  we have $[\mathrm{Der}_1^0\,\mathfrak{n}_{2,t},\mathrm{ad}_{\mathfrak{n}_{2,t}}\mathfrak{t}_0]=0$, so the set of derivations $\mathrm{ad}_{\mathfrak{n}_{2,t}}\mathfrak{t}_0$ centralizes $\mathrm{Der}_1^0\,\mathfrak{n}_{2,t}$ in $\mathrm{Der}\, \mathfrak{n}_{2,t}$. All these basic ideas yield to the following result:

\begin{lem}\label{nonsolvext}Let $\mathfrak{g}$ be a nonsolvable Lie algebra with nilradical $\mathfrak{n}_{2,t}$. Then $\mathfrak{g}$ is one of the following Lie algebras:
\begin{itemize}
\item [a)] $\mathfrak{g}=\mathfrak{s}\oplus \mathfrak{r}$ is a direct sum as ideals of a semisimple algebra $\mathfrak{s}$ and a solvable Lie algebra $\mathfrak{r}$ with nilradical $\mathfrak{n}_{2,t}$;
\item [b)] $\mathfrak{g}=\mathfrak{s}\oplus\mathfrak{g}_{2,t}$, a direct sum as ideals where $\mathfrak{s}$ is either the null algebra or any arbitrary semisimple algebra or
\item [c)] $\mathfrak{g}=\mathfrak{s}\oplus \mathfrak{g}(\delta)$, a direct sum as ideals where $\mathfrak{s}$ is either the null algebra or any arbitrary semisimple algebra and $\mathfrak{g}(\delta)$ is the Lie algebra $\mathfrak{g}(\delta)=\mathfrak{g}_{2,t}\oplus \mathbb{C}\cdot x$, with $[\mathrm{Der}_1^0\,\mathfrak{n}_{2,t},x]=0$ and $ad_{\mathfrak{n}_{2,t}}x=id_{2,t}+\delta$,  $\delta \in\mathfrak{N}_{2,t}$ where $\delta\in C_{\mathrm{Der}\,\mathfrak{n}_{2,t}}(\mathrm{Der}_1^0\,\mathfrak{n}_{2,t})$.
\end{itemize}
\end{lem}

\begin{proof}  Let $\mathfrak{g}=\mathfrak{s}\oplus \mathfrak{r}$ a Levi decomposition of $\mathfrak{g}$. In case $[\mathfrak{s},\mathfrak{r}]=0$ or $\mathfrak{t}_0= 0$ we get items $a)$ and $b)$. Otherwise, we have $\mathfrak{t}_0\neq0$ and $[\mathfrak{s},\mathfrak{n}_{2,t}]\neq 0$ and, from Theorem \ref{free nilpotent} and preliminary comments in this section, $\mathfrak{g}$ decomposes as a direct sum of ideals, $\mathfrak{g}=\mathfrak{s}_1\oplus (\mathrm{Der}_1^0\,\mathfrak{n}_{2,t}\oplus_{id}\mathfrak{n}_{2,t}\oplus \mathfrak{t}_0)$, $\mathfrak{s}_1$ being either semisimple or $\mathfrak{s}_1=0$. Let $\mathfrak{r}_0$ the subalgebra defined in (\ref{zeromod}) and consider the (restricted) adjoint representation $\mathrm{ad}_{\mathfrak{n}_{2,t}}: \mathfrak{r}_0\to \mathrm{Der}\, \mathfrak{n}_{2,t}$, which is both an homomorphism of Lie algebras and a $\mathrm{Der}_1^0\,\mathfrak{n}_{2,t}$-module homomorphism. Note that $\mathrm{Ker}\, \mathrm{ad}_{\mathfrak{n}_{2,t}}=C_{\mathfrak{r}_0}(\mathfrak{n}_{2,t})=\mathcal{Z}(\mathfrak{n}_{2,t})_0$, i.e. the trivial module living inside the center of the nilradical $\mathfrak{n}_{2,t}$. Decompose $(\mathfrak{n}_{2,t})_0$ as a direct sum of subspaces, $(\mathfrak{n}_{2,t})_0=\mathcal{Z}(\mathfrak{n}_{2,t})_0\oplus \mathfrak{v}_0$. The quotient Lie algebra $\displaystyle\frac{\mathfrak{r}_0}{\mathcal{Z}(\mathfrak{n}_{2,t})_0}\ (\cong \mathfrak{v}_0\oplus \mathfrak{t}_0$) is isomorphic to $\mathrm{ad}_{\mathfrak{n}_{2,t}}(\mathfrak{r}_0)$, a subalgebra of $\mathrm{Der}\, \mathfrak{n}_{2,t}$ contained in $C_{\mathrm{Der}\,\mathfrak{n}_{2,t}}(\mathrm{Der}_1^0\,\mathfrak{n}_{2,t})=k\cdot id_{2,t}\oplus (\mathfrak{N}_{2,t})_0$, the last sumand is the sum of trivial $\mathrm{Der}_1^0\,\mathfrak{n}_{2,t}$-modules inside $\mathfrak{N}_{2,t}$ via the adjoint representation $[\delta,\mu]=\delta\mu-\mu\delta$. But $\mathrm{ad}_{\mathfrak{n}_{2,t}^2}\mathfrak{t}_0$ has no nilpotent derivations (the nilradical of $\mathfrak{g}$ is just $\mathfrak{n}_{2,t}$), so none of the elements inside $\mathrm{ad}_{\mathfrak{n}_{2,t}}(\mathfrak{t}_0)$ belongs to $\mathfrak{N}_{2,t}$. Let $x,y\in \mathfrak{t}_0$ be two nonzero elements. Rescaling, we can assume $ad\,x=id_{2,t} \oplus \delta$ and $ad\,y=id_{2,t} \oplus \mu$, where $\delta, \mu\in (\mathfrak{N}_{2,t})_0$. Then $ad\, (x-y)\in (\mathfrak{N}_{2,t})_0$ and therefore, $x-y$ is a nilpotent element in $\mathfrak{t}_0$; the only possibility is $x=y$ and part c) follows. It is immediate to check that items a), b) and c) provides Lie algebras.
\end{proof}

For free nilpotent of low index, the centralizer $C_{\mathrm{Der}\,\mathfrak{n}_{2,t}}(\mathrm{Der}_1^0\,\mathfrak{n}_{2,t})$ is easy to be computed. In fact, we have:
\begin{itemize}
\item $t=1,2$: $C_{\mathrm{Der}\,\mathfrak{n}_{2,t}}(\mathrm{Der}_1^0\,\mathfrak{n}_{2,t})=\mathbb{C}\cdot id_{2,t}$;
\item $t=3,4$: $C_{\mathrm{Der}\,\mathfrak{n}_{2,t}}(\mathrm{Der}_1^0\,\mathfrak{n}_{2,t})=\mathbb{C}\cdot id_{2,t}\oplus \mathbb{C}\cdot  \mathrm{ad}\, \omega_0$, where $\omega_0=[v_0,v_1]$ according to Proposition \ref{low free nilpotent}.
\end{itemize}
 The above centralizer computation and Lemma  \ref{nonsolvext} allow us to stablish:

\begin{thm} For $t=1,2,3,4$ and up to isomorphisms, the faithful nonsolvable complex Lie algebras with nilradical $\mathfrak{n}_{2,t}$  are $\mathfrak{g}_{2,t}=\mathrm{Der}_1^0\,\mathfrak{n}_{2,t}\oplus_{id}\mathfrak{n}_{2,t}$ described in Theorem \ref{free nilpotent} and those of the form $\mathfrak{g}=\mathfrak{g}_{2,t}\oplus k\cdot x$ where $ad_{\mathfrak{g}}\, x$ acs trivially on $\mathrm{Der}_1^0\,\mathfrak{n}_{2,t}$ and as the $id_{2,t}$ on $\mathfrak{n}_{2,t}$. 
\end{thm}
\begin{proof}
The result follows trivially from Lemma \ref{nonsolvext} in cases $t=1,2$, and for $t=2,3$ we note that  the elements inside $C_{\mathrm{Der}\,\mathfrak{n}_{2,t}}(\mathrm{Der}_1^0\,\mathfrak{n}_{2,t})$ are of the form $\alpha\cdot id_{2,t}\oplus \beta\cdot  \mathrm{ad}\, \omega_0$. But  $\mathrm{ad}\, \omega_0$ is an inner derivation of $\mathfrak{n}_{2,t}$ and therefore it can be removed.
\end{proof}


\section{Conclusion: Lie algebras with nilradical $\mathfrak{n}_{2,t}$ for $t=1,2,3$}

From previous results and comments, the classification of Lie algebras with nilradical $\mathfrak{n}_{2,t}$, boils down to the determination of suitable derivation sets containing no nilpotent elements as described in Lemmas \ref{solvExt} and \ref{nonsolvext}. The sets of derivations has been completely determined for $t=1,2,3$ in section 3 (solvable extensions) by using explicit matrix descriptions of derivations, inner derivations and automorphisms of $\mathfrak{n}_{2,t}$ and in section 4 (nonsolvable extensions) by computing centralizers inside $\mathrm{Der}\, \mathfrak{n}_{2,t}$. In this final section we summarized the results in a unique theorem, where the whole list of Lie algebras with nilradical $\mathfrak{n}_{2,1}, \mathfrak{n}_{2,2}$ or $\mathfrak{n}_{2,3}$ is given through their multiplication tables and the nested basis descriptions obtained in Proposition \ref{low free nilpotent}:

\begin{thm}\label{conclusion} Up tp isomorphisms, the Lie algebras with nilradical a free nilpotent Lie algebra of type 2 and nilindex $t\leq 3$ are (only nonzero products involving elements not in the nilradical are given):
\begin{itemize}

\item[i)] The abelian 2-dimensional $\mathfrak{n}_{2,1}$;
\item[ii)] $\mathfrak{r}_{2,1}^{1}=\mathfrak{n}_{2,1}\oplus \mathbb{C}\cdot x:\quad$ $[x,v_0]=v_0+v_1, [x,v_1]=v_1$;
\item[iii)]$\mathfrak{r}_{2,1}^{1,\alpha}=\mathfrak{n}_{2,1}\oplus \mathbb{C}\cdot x:\quad$$[x,v_0]=v_0$ and $[x,v_1]=\alpha v_1$, $\alpha \in \mathbb{C}$;
\item[iv)] $\mathfrak{r}_{2,1}^{2}=\mathfrak{n}_{2,1}\oplus \mathbb{C}\cdot x\oplus \mathbb{C}\cdot y:\quad$ $[x,v_0]=[y,v_0]=v_0, [x,v_1]=-[y,v_1]=v_1$;
\item[v)]$\mathfrak{g}_{2,1}=\mathfrak{sl}_2(\mathbb{C})\oplus\mathfrak{n}_{2,1}:\quad$$[e,f]=h$, $[h,e]=2e$, $[h,f]=-2f$, $[h,v_0]=v_0$, $[h,v_1]=-v_1$, $[e,v_1]=v_0$, $[f,v_0]=v_1$;
\item[vi)]$\mathfrak{g}_{2,1}^1=\mathfrak{sl}_2(\mathbb{C})\oplus\mathfrak{r}_{2,1}^{1,1}:\quad$$[e,f]=h$, $[h,e]=2e$, $[h,f]=-2f$, $[h,v_0]=v_0$, $[h,v_1]=-v_1$, $[e,v_1]=v_0$, $[f,v_0]=v_1$, $[x,v_0]=v_0$, $[x,v_1]=v_1$;
\item[vii)]The Heisenberg 3-dimensional $\mathfrak{n}_{2,2}$: $[v_0,v_1]=w_0$,
\item[viii)]$\mathfrak{r}_{2,2}^{1}=\mathfrak{n}_{2,2}\oplus \mathbb{C}\cdot x:\quad$ $[x,v_0]=v_0+v_1, [x,v_1]=v_1$, $[x,w_0]=2w_0$;
\item[ix)]$\mathfrak{r}_{2,2}^{1,\alpha}=\mathfrak{n}_{2,2}\oplus \mathbb{C}\cdot x:\quad$$[x,v_0]=v_0$, $[x,v_1]=\alpha v_1$, $[x,w_0]=(1+\alpha) w_0$, $\alpha \in \mathbb{C}$;
\item[x)] $\mathfrak{r}_{2,2}^{2}=\mathfrak{n}_{2,2}\oplus \mathbb{C}\cdot x\oplus \mathbb{C}\cdot y:\quad$$[x,v_0]=[y,v_0]=v_0$, $[x,v_1]=-[y,v_1]=v_1$, $[x,w_0]=2w_0$;
\item[xi)]$\mathfrak{g}_{2,2}=\mathfrak{sl}_2(\mathbb{C})\oplus\mathfrak{n}_{2,2}:\quad$$[e,f]=h$, $[h,e]=2e$, $[h,f]=-2f$, $[h,v_0]=v_0$, $[h,v_1]=-v_1$, $[e,v_1]=v_0$, $[f,v_0]=v_1$;
\item[xii)]$\mathfrak{g}_{2,2}^1=\mathfrak{sl}_2(\mathbb{C})\oplus\mathfrak{r}_{2,2}^{1,1}:\quad$$[e,f]=h$, $[h,e]=2e$, $[h,f]=-2f$, $[h,v_0]=v_0$, $[h,v_1]=-v_1$, $[e,v_1]=v_0$, $[f,v_0]=v_1$, $[x,v_0]=v_0$, $[x,v_1]=v_1$; $[x,w_0]=2w_0$.
\item[xiii)]The $5$-dimensional $\mathfrak{n}_{2,3}$: $[v_0,v_1]=w_0$, $[v_i,w_0]=z_i, i=0,1$,
\item[xiv)]$\mathfrak{r}_{2,3}^{1}=\mathfrak{n}_{2,3}\oplus \mathbb{C}\cdot x:\quad$ $[x,v_0]=v_0+v_1$, $[x,v_1]=v_1$, $[x,w_0]=2w_0$, $[x,z_0]=3z_0+z_1$, $[x,z_1]=3z_1$;
\item[xv)]$\mathfrak{r}_{2,3}^{1,\alpha}=\mathfrak{n}_{2,3}\oplus \mathbb{C}\cdot x:\quad$ $[x,v_0]=v_0$, $[x,v_1]=\alpha v_1$,$[x,w_0]=(1+ \alpha)w_0$, $[x,z_0]=(2+ \alpha)z_0$,$[x,z_1]=(1+ 2\alpha)z_1$, $\alpha \in \mathbb{C}$;
\item[xvi)]$\mathfrak{r}_{2,3}^{2}=\mathfrak{n}_{2,3}\oplus \mathbb{C}\cdot x:\quad$$[x,v_0]=v_0+z_1$, $[x,w_0]=w_0$, $[x,z_0]=2z_0$, $[x,z_1]=z_1$;
\item[xvii)]$\mathfrak{r}_{2,3}^{3}=\mathfrak{n}_{2,3}\oplus \mathbb{C}\cdot x:\quad$ $[x,v_0]=v_0+z_0$, $[x,v_1]=-v_1$, $[x,z_0]=z_0$, $[x,z_1]=-z_1$;
\item[xviii)]$\mathfrak{r}_{2,3}^{4}=\mathfrak{n}_{2,3}\oplus \mathbb{C}\cdot x \oplus \mathbb{C}\cdot y:\quad$$[x,v_0]=[y,v_0]=v_0$, $[x,v_1]=-[y,v_1]=v_1$, $[x,w_0]=2w_0$, $[x,z_0]=3z_0$, $[x,z_1]=3z_1$, $[y,z_0]=z_0$, $[y,z_1]=-z_1$;
\item[xix)]$\mathfrak{g}_{2,3}=\mathfrak{sl}_2(\mathbb{C})\oplus\mathfrak{n}_{2,3}:\quad$$[e,f]=h$, $[h,e]=2e$, $[h,f]=-2f$, $[h,v_0]=v_0$, $[h,v_1]=-v_1$, $[e,v_1]=v_0$, $[f,v_0]=v_1$, $[h,z_0]=z_0$, $[h,z_1]=-z_1$,$[e,z_1]=z_0$, $[f,z_0]=z_1$, $[v_0,v_1]=w_0$, $[v_0,w_0]=z_0$ and $[v_1,w_0]=z_1$;
\item[xx)]$\mathfrak{g}_{2,3}^{1}=\mathfrak{sl}_2(\mathbb{C})\oplus\mathfrak{r}_{2,3}^{1,1}:$ $[e,f]=h$, $[h,e]=2e$, $[h,f]=-2f$, $[h,v_0]=v_0$, $[h,v_1]=-v_1$, $[e,v_1]=v_0$, $[f,v_0]=v_1$, $[h,z_0]=z_0$, $[h,z_1]=-z_1$,$[e,z_1]=z_0$, $[f,z_0]=z_1$, $[v_0,v_1]=w_0$, $[v_0,w_0]=z_0$, $[v_1,w_0]=z_1$, $[x,v_0]=v_0$, $[x,v_1]= v_1$, $[x,w_0]=2 w_0$,$[x,z_0]=3 z_0$, $[x,z_1]=3 z_1$; 
\item[xxi)]the direct sum as ideals $\mathfrak{s}\oplus \mathfrak{g}$, where $\mathfrak{s}$ is a semisimple Lie algebra and $ \mathfrak{g}$ one of the Lie algebras given in any of the previous items $i)-xx)$.
\end{itemize}
\end{thm}
\begin{rmk} Theorem \ref{conclusion} provides the complete classification of complex Lie algebras with nilradical the $3$-dimensional Heisenberg algebra $\mathfrak{n}_{2,2}$; the classification problem in case of solvable extensions was solved in \cite{RuWi93} even for the real field. We note that the technics we are used here can also be extended to provide the analogous classification problem over the real field for $\mathfrak{n}_{2,t}$, $t=1,2,3,4$.
\end{rmk}

\begin{rmk} Extensions of $\mathfrak{n}_{2,3}$ are treated in \cite[Section 3, subsection 3.2]{An11}. In this paper, $\mathfrak{n}_{2,3}$ is denoted as $\mathcal{L}_{5,3}$ and introduced in  Theorem 1, subsection 2.2, through the multiplication table $[X_0,X_1]=X_2, [X_0,X_2]=X_3, [X_1,X_2]=X_4$; we note that the basis $\{X_0,X_1,X_2,X_3,X_4\}$ is just our basis $\{v_0,v_1,w_0,z_1,z_2\}$. The extended Lie algebras appear listed in Propositions 2, 3 and 4. Proposition 2 provides the $1$-solvable extensions $\mathfrak{g}_{5,3}^{6,1}$, $\mathfrak{g}_{5,3}^{6,2}$, $\mathfrak{g}_{5,3}^{6,3}$ and $\mathfrak{g}_{5,3}^{6,4}$. The first algebra and the third one are just $\mathfrak{r}_{2,3}^{1,\alpha}$ and $\mathfrak{r}_{2,3}^{2}$. The second algebra is equal to $\mathfrak{r}_{2,3}^{3}$ by  rescaling the basis of $\mathfrak{g}_{5,3}^{6,2}$ given \cite{An11}  in the folowing way: $\{-Y; X_0,X_1,X_2,X_3,-X_4\}$. The multiplication table for $\mathfrak{g}_{5,3}^{6,4}$ has a misprint, $[Y,X_4]=3X_4$ must be declared instead of $4X_4$ (otherwise, $\mathrm{ad}\, Y$ is not a derivation, in other words, Jacobi identity $J(a,b,c)=[[a,b],c] +[[a,b],c]+[[a,b],c]=0$ fails for the terna $(a,b,c)=(Y,X_1,X_2)$); by doing the correctio, $\mathfrak{g}_{5,3}^{6,4}$ is just $\mathfrak{r}_{2,3}^{1}$. Proposition 3 provides the $2$-sovable extension $\mathfrak{g}_{5,3}^{7,1}$ which  is equal to our $\mathfrak{r}_{2,3}^{4}$ if we consider the new basis $\{Y+Z,Y-Z; X_0,X_1,X_2,X_3,X_4\}$.

In Proposition 4, the authors give the nonsolvable extensions $\mathfrak{g}_{5,3}^{8,1}$ and $\mathfrak{g}_{5,3}^{9,1}$. The former algebra is just our $\mathfrak{g}_{2,t}$ taking in account that $\{Y,Z,X\}$ is a standard basis of the split $3$-dimensional simple Lie algebra ($\{h,e,f\}$ with $[h,e]=2e, [h,f]=-2f,[e,f]=h$). But $\mathfrak{g}_{5,3}^{9,1}$ has been erroneously included in the classification, this is not a Lie algebra: $J(Y-Y',Z,X_1)=-X_0$ and $J(Y-Y',Z,X_4)=-X_3$. In this case we have not a misprint but an structural error: $\mathfrak{s}=\mathrm{span}\langle Y-Y',Z,X\rangle$ is just a $3$-split simple Lie algebra and $\mathcal{L}_{5,3}$ is $\mathrm{ad}\, \mathfrak{s}$-invariant, so $\mathcal{L}_{5,3}$ decomposes as a finite sum of $\mathfrak{s}$-irreducible modules; but the Cartan subalgebra $\mathbb{C}\cdot (Y-Y')$ acts  with $5$ different eigenvalues on $\mathcal{L}_{5,3}$ (exactly $1,-2,-1,0,-3$) with corresponding eigenvectors $X_0,X_1,X_2,X_3,X_4$; this fails in any $\mathfrak{s}$-module decomposition of $\mathcal{L}_{5,3}$. The final algebra that left in the classification is our Lie algebra $\mathfrak{g}_{2,3}^{1}$ described in Theorem \ref{conclusion},  item $xx)$.
\end{rmk}

\section*{Acknowledgements}
The authors would like to thank Spanish Government project  MTM 2010-18370-C04-03. Daniel de-la-Concepci\'on also thanks support from Spanish FPU grant 12/03224.

\end{document}